\documentclass[11 pt]{article}

\usepackage{ amsfonts,amscd,amssymb,amsthm}
\usepackage{amsfonts}
\usepackage{amssymb}
\usepackage{amsmath}
\usepackage{amsxtra}
\usepackage{amscd}
\usepackage{amsthm}
\usepackage{eucal}
\usepackage{array}
\usepackage{graphicx}
\usepackage{tikz}
\usepackage{enumerate}
\newtheorem{theorem}{Theorem}[section]
\newtheorem{lemma}[theorem]{Lemma}
\newtheorem{proposition}[theorem]{Proposition}
\newtheorem{corollary}[theorem]{Corollary}

\theoremstyle{remark}
\newtheorem{remark}[theorem]{Remark}

\theoremstyle{definition}
\newtheorem{example}[theorem]{Example}
\newtheorem{definition}[theorem]{Definition}

\title{Multigraded Betti numbers of some path ideals}
\author{Nursel Erey \footnote{Department of Mathematics and Statistics, Dalhousie University, Halifax, Nova Scotia, Canada B3H 4R2 e-mail: nurselerey@gmail.com}}
\date{\vspace{-5ex}} 
\begin{document}
\maketitle


\begin{abstract}
We determine (multi)graded Betti numbers of path ideals of lines and star graphs.
\end{abstract}

\section{Introduction}

The path ideal of a directed graph was introduced by Conca and De Negri \cite{conca} and recently these ideals have been studied by many authors, see \cite{alilooee faridi 2, alilooee faridi, banerjee, bouchat ha okeefe, morey et al, he van tuyl, Kubitzke Olteanu, kubik}. In this paper we consider the path ideals of undirected graphs. In particular, we give formulas for Betti numbers of path ideals of lines and stars extending the work of \cite{alilooee faridi 2}.

\section{Preliminaries}
\subsection{Simplicial complexes and homology}
An \textbf{abstract simplicial complex} $\Delta$ on a set of \textbf{vertices} $V(\Delta)=\{v_1,...,v_n\}$ is a collection of subsets of $V$ such that $\{v_i\}\in \Delta$ for all $i$ and, $F\in \Delta$ implies that all subsets of $F$ are also in $\Delta .$ The elements of $\Delta$ are called \textbf{faces} and the maximal faces under inclusion are called \textbf{facets}. If the facets $F_1,...,F_q$ generate $\Delta$, we write $\Delta= \langle F_1,...,F_q\rangle$ or $\operatorname{Facets}(\Delta)=\{F_1,...,F_q\}$.

A face $\{v_1,v_2,...,v_n\}-\{v_{i_1},...,v_{i_s}\}$ will be denoted by $\{v_1,...,\widehat{v}_{i_1},...,\widehat{v}_{i_s},...,v_n  \}$ for $i_1 < i_2 < ....< i_s.  $ 

Two simplicial complexes $\Delta$ and $\Gamma$ are \textbf{isomorphic} if there is a bijection $\varphi:V(\Delta)\rightarrow V(\Gamma)$ between their vertex sets such that $F$ is a face of $\Delta$ iff $\varphi(F)$ is a face of $\Gamma$.

Let $\Delta$ and $\Gamma$ be simplicial complexes which has no common vertices. Then the \textbf{join} of $\Delta$ and $\Gamma$ is the simplicial complex given by $$ \Delta \ast \Gamma =\{\delta \cup \gamma : \delta \in \Delta, \gamma \in \Gamma  \}. $$

A \textbf{cone} with \textbf{apex} $v$ is a special join obtained by joining a simplicial complex $\Delta$ with $\{\emptyset, v\}$ where $v$ is an element which is not in the vertex set of $\Delta$. Equivalently, a simplicial complex is a cone with apex $v$ if $v$ is a member of every facet.

For each integer $ i $, the $ \Bbbk $ -vector space $ \widetilde{H}_i(\Delta,\Bbbk) $ is the $ i^{th} $ \textbf{reduced homology} of $ \Delta $ over $ \Bbbk $. For the sake of simplicity, we drop $\Bbbk $ and write $\widetilde{H}_i(\Delta)$ whenever we work on a fixed ground field $\Bbbk$.

A \textbf{simplex} is a simplicial complex that contains all subsets of its nonempty vertex set. The \textbf{boundary} $\Sigma$ of a simplex $\Delta=\langle \{v_1,...,v_n\}\rangle$ is obtained from $\Delta$ by removing the maximal face of $\Delta$. And, the homology groups of $\Sigma$ are given by
\begin{equation}\label{homology of boundary}
    \widetilde{H}_{p}(\Sigma, \Bbbk) \cong
    \begin{cases}
       \Bbbk & \text{if } p = n-2   \\
      0 & \text{ otherwise. } 
    \end{cases}
  \end{equation}

The \textbf{irrelevant complex} $\{\emptyset\}$ has the homology groups
\begin{equation}\label{homology of irrelevant complex}
    \widetilde{H}_{p}(\{\emptyset\}, \Bbbk) \cong
    \begin{cases}
       \Bbbk & \text{if } p = -1   \\
      0 & \text{ otherwise. } 
    \end{cases}
  \end{equation}
whereas the \textbf{void complex} $\{\}$ has trivial reduced homology in all degrees.

A simplicial complex $\Delta$ is \textbf{acyclic} (over $\Bbbk$) if $ \widetilde{H}_i(\Delta,\Bbbk) $ is trivial for all $i$. Examples of acyclic complexes include cones and simplices.

The homology of two simplicial complexes is related to homology of their union and intersection by the Mayer-Vietoris long exact sequence:

\begin{theorem}[Corollary 6.4, \cite{Rotman algebraic topology}]\label{Mayer-Vietoris theorem} Let $\Delta_1$ and $\Delta_2$ be two simplicial complexes. Then there is a long exact sequence
\begin{equation}\label{MV sequence} \cdots \rightarrow \widetilde{H}_{p}(\Delta_1)\oplus \widetilde{H}_{p}(\Delta_2) \rightarrow \widetilde{H}_{p}(\Delta_1 \cup \Delta_2) \rightarrow \widetilde{H}_{p-1}(\Delta_1 \cap \Delta_2) \rightarrow \widetilde{H}_{p-1}(\Delta_1)\oplus \widetilde{H}_{p-1}(\Delta_2)\rightarrow \cdots
\end{equation}
where the homology can be taken over any field.
\end{theorem}
A particular case of Theorem \ref{Mayer-Vietoris theorem} occurs when a simplicial complex $\Delta=\Delta_1 \cup \Delta_2$ is a union of two acyclic subcomplexes $\Delta_1$ and $\Delta_2$. In that case, the sequence (\ref{MV sequence}) becomes
$$\cdots \rightarrow 0 \rightarrow \widetilde{H}_{p}(\Delta_1 \cup \Delta_2) \rightarrow \widetilde{H}_{p-1}(\Delta_1 \cap \Delta_2) \rightarrow 0 \rightarrow \cdots $$
whence $\widetilde{H}_{p}(\Delta_1 \cup \Delta_2)$ and $\widetilde{H}_{p-1}(\Delta_1 \cap \Delta_2)$ are isomorphic for all $p$. Since we will make frequent use of this specific case we state it separately as an immediate Corollary.

\begin{corollary}\label{Corollary to MV acyclic union}If $\Delta_1$ and $\Delta_2$ are acyclic simplicial complexes then  $$\widetilde{H}_{p}(\Delta_1 \cup \Delta_2) \cong \widetilde{H}_{p-1}(\Delta_1 \cap \Delta_2) $$
for every $p$ and the homology can be taken over any field.
\end{corollary}
\subsection{Graphs and resolutions}

Let $S=\Bbbk[x_1,...,x_n]$ be the polynomial ring in $n$ variables over a field $\Bbbk$. Given a minimal multigraded free resolution 
$$ 0 \longrightarrow \bigoplus_{\bold{m} \in \mathbb{N}^n} S(-\bold{m})^{b_{r,\bold{m}}(I)} \overset{\partial_r}{\longrightarrow} ...\longrightarrow \bigoplus_{\bold{m} \in \mathbb{N}^n} S(-\bold{m})^{b_{1,\bold{m}}(I)} \overset{\partial_1}{\longrightarrow} \bigoplus_{\bold{m} \in \mathbb{N}^n} S(-\bold{m})^{b_{0,\bold{m}}(I)} \overset{\partial_0}{\longrightarrow} I \longrightarrow 0   $$
of $I$, the associated ranks $b_{i,\bold{m}}(I)$ are called \textbf{multigraded Betti numbers} of $I$. Graded and multigraded Betti numbers are related by the equation
\begin{equation}\label{definition of graded betti number} 
b_{i,j}(I)=\sum_{\deg(m)=j}b_{i,\bold{m}}(I) 
\end{equation} 
where $\deg(m)$ stands for the standard degree of $m$ (i.e
$\deg(x_1^{a_1}...x_n^{a_n})=a_1+...+a_n).$

For a \textbf{graph} $G$, the vertex and edge sets are denoted by $V(G)$ and $E(G)$ respectively. All graphs in this paper should be assumed simple meaning that loopless and without multiple edges. Two vertices $u$ and $v$ are \textbf{adjacent} to one another if $\{u,v\}$ is an edge of $G$. A vertex $u$ of $G$ is called an \textbf{isolated vertex} if it is not adjacent to any vertex of $G$. We will say that $G$ is of \textbf{size} $e$ and of \textbf{order} $n$ if it has $e$ edges and $n$ vertices. For two vertices $u$ and $v$ of $G$, a \textbf{path} of length $t-1$ from $u$ to $v$ is a sequence of $t\geq 2$ distinct vertices $u=z_1,...,z_t=v$ such that $\{z_i, z_{i+1}\}\in E(G)$ for all $i=1,...,t-1$. We will denote by $L_n$ a line of order $n$. Also $C_n$ and $\mathcal{S}_n$ will be a cyle and a star graph of size $n$ respectively.

\begin{figure}[ht]
\centering
\begin{minipage}[b]{0.32\linewidth}
\centering
\begin{tikzpicture}[scale=4.2]
\tikzstyle{every node}=[circle, fill=black!,inner sep=0pt, minimum
width=4pt]

 \node (n_1) at (-.2,.2)[label=left:{}] {}; 
 \node (n_2) at (.2,.2)[label=right:{}] {};
 \node (n_3) at (-.2,-.2)[label=left:{}] {};
 \node (n_4) at (.2,-.2)[label=right:{}] {};

\foreach \from/\to in
  {n_1/n_2, n_1/n_3, n_3/n_4}\draw (\from) -- (\to);
  
\end{tikzpicture}
\caption{$L_4$}
\label{fig:minipage1}
\end{minipage}
\begin{minipage}[b]{0.32\linewidth}
\centering
\begin{tikzpicture}[scale=4.2]
\tikzstyle{every node}=[circle, fill=black!,inner sep=0pt, minimum
width=4pt]

 \node (n_1) at (-.2,.2)[label=left:{}] {}; 
 \node (n_2) at (.2,.2)[label=right:{}] {};
 \node (n_3) at (-.2,-.2)[label=left:{}] {};
 \node (n_4) at (.2,-.2)[label=right:{}] {};

\foreach \from/\to in
  {n_1/n_2, n_1/n_3, n_2/n_4, n_3/n_4}\draw (\from) -- (\to);
  
\end{tikzpicture}

\caption{$C_4$}
\label{fig:minipage2}
\end{minipage}
\begin{minipage}[b]{0.32\linewidth}
\centering
\begin{tikzpicture}[scale=4.2]
\tikzstyle{every node}=[circle, fill=black!,inner sep=0pt, minimum
width=4pt]

 \node (n_1) at (-.2,.2)[label=left:{}] {}; 
 \node (n_2) at (.02,.16)[label=right:{}] {};
 \node (n_3) at (-.2,-.2)[label=left:{}] {};
 \node (n_4) at (.2,-.2)[label=right:{}] {};
 \node (n_5) at (.16,.02)[label=right:{}] {};

\foreach \from/\to in
  {n_3/n_2, n_1/n_3, n_3/n_4, n_3/n_5}\draw (\from) -- (\to);
  
\end{tikzpicture}
\caption{$\mathcal{S}_4$}
\label{fig:minipage3}
\end{minipage}
\end{figure}

If $G$ is a graph with vertex set $V=\{x_1,...,x_n \}$ then its \textbf{path ideal} $I_t(G)$ is the monomial ideal of $S=\Bbbk[x_1,...x_n]$ given by
$$I_t(G)= ( x_{i_1}...x_{i_t} \mid x_{i_1},..., x_{i_t} \text{ is a path of length } t-1 \text{ in } G ). $$
\begin{example}
A graph  $G$ of order $4$ which has the path ideals $I_4(G)=(x_2x_1x_4x_3),I_3(G)=(x_2x_1x_4, x_2x_1x_3, x_1x_3x_4)$, $I_2(G)=( x_1x_2, x_1x_3, x_1x_4, x_3x_4 ), I_1(G)=(x_1, x_2, x_3, x_4). $
\end{example}
\begin{figure}[htp]
\centering
\begin{tikzpicture}[scale=4.2]
\tikzstyle{every node}=[circle, fill=black!,inner sep=0pt, minimum
width=4pt]

 \node (n_1) at (-.2,.2)[label=left:{$x_1$}] {};
 \node (n_2) at (.2,.2)[label=right:{$x_2$}] {};
 \node (n_3) at (-.2,-.2)[label=left:{$x_3$}] {};
 \node (n_4) at (.2,-.2)[label=right:{$x_4$}] {};

\foreach \from/\to in
  {n_1/n_2, n_1/n_3, n_1/n_4, n_3/n_4}\draw (\from) -- (\to);
  
\end{tikzpicture}\end{figure}
For a square-free monomial $m$ we denote by $G_m$ the \textbf{induced subgraph} of $G$ on the set of vertices that divide $m$.

Let $I=(m_1,...,m_s)$ be a monomial ideal of $S$ which is minimally generated by the set of monomials $M=\{m_1,...,m_s\}$. The \textbf{Taylor simplex} $\Theta$ of $I$ is a simplex on $s$ vertices which are labelled with the minimal generators of $I$. If $\tau=\{m_{i_1},...,m_{i_r}\}$ is a face of $\Theta$, then by $\operatorname{lcm}(\tau)$ we mean $\operatorname{lcm}(m_{i_1},...,m_{i_r})$.  For any monomial $m$ in $S$, 
$$ \Theta_{\leq m}= \{\tau \in \Theta \mid \operatorname{lcm}(\tau) \text{ divides } m\} $$ 
and
$$ \Theta_{<m}= \{\tau \in \Theta \mid \operatorname{lcm}(\tau) \text{ strictly divides } m\} $$
are subcomplexes of $\Theta$. Clearly we have the equation
$$\Theta_{<m}= \bigcup_{x_i \vert m}\Theta_{\leq \frac{m}{x_i}} $$
and every facet of $\Theta_{\leq \frac{m}{x_i}}$ is of the form
$$F_i:= V(\Theta_{\leq m})- \{u\in M \mid x_i \text{ does not divide } u \}.$$ Therefore we have
\begin{equation}\label{facets of theta less than m}
F_i \in \operatorname{Facets}(\Theta_{<m}) \Leftrightarrow F_i \text{ is maximal in } \{F_j \mid x_j \text{ divides } m \}.
\end{equation}
The following Theorem will be our main tool to calculate Betti numbers in this paper.
\begin{theorem}[\cite{bayer peeva sturmfels}]\label{betti number by homology} Let I be a monomial ideal of $S$ which is minimally generated by the monomials $m_1,...,m_s$.  Denote by $\Theta$ the Taylor simplex of $I$. For $i \geq 1$, the multigraded Betti numbers of $S/I$ are given by
 \begin{equation}
     b_{i,m} (S/I)=
    \begin{cases}
      \dim_{\Bbbk} \widetilde{H}_{i-2}(\Theta_{<m};\Bbbk) & \text{if}\ m \text{ divides }\ \operatorname{lcm}(m_1,...m_s)\\
      0 & \text{otherwise.}
    \end{cases}
  \end{equation}
\end{theorem}

\begin{remark} If $I=(m_1,...,m_s)$ and $q=\deg\operatorname{lcm}(m_1,...,m_s)$ then for any $r > q$ we have $b_{i,r}(I)=0$ for all $i$.  Therefore we call the numbers $b_{i,q}(I), i \in \mathbb{Z}$ as the \textbf{top grade Betti numbers}.
\end{remark}
\begin{remark}\label{isolated vertex}Suppose that $\Delta$ is the Taylor simplex of $I(G)$ for some graph $G$. If the induced graph $G_m$ contains an isolated vertex, then $\Delta_{<m}=\Delta$ is a simplex. So $b_{i,m}(S/I(G))=0$ for all $i$ by Theorem \ref{betti number by homology}.
\end{remark}

\begin{lemma} [\cite{erey faridi}] If $I_1, I_2,...,I_N$ are square-free monomial ideals whose minimal generators contain no common variable and each $I_i$ has minimal generators whose least common multiple is of degree $q_i$, then
\begin{equation}\label{top betti numbers of connected components}
b_{i, q_1+...+q_N}(S/(I_1+I_2+...+I_N))=  \sum_{u_1+...+u_N=i } b_{u_1,q_1}(S/I_1).... b_{u_N,q_N}(S/I_N).
\end{equation}
\end{lemma}
\begin{lemma}\label{multigraded betti number and induced subgraph} If $m$ is a square-free monomial of degree $j$ and $t\geq 2$, then $b_{i,m}(S/I_t(G))=b_{i,j}(S/I_t(G_m))$.
\end{lemma}
\begin{proof}
Proof is similar to Lemma 3.1 in \cite{erey faridi}.
\end{proof}
\section{Betti numbers of some path ideals}
\begin{definition}\label{definition useful simplicial complex} For any $n\geq t\geq 1$ the simplicial complex $\Omega^n_t$ on the set of vertices  $\{1,...,n\}$ is defined by
$$ \operatorname{Facets}(\Omega_t^n)=\big\{ \{1,...,\hat{i},\widehat{i+1},...,\widehat{i+t-1},i+t,...,n\} |  \ i=1,...,n-t+1   \big\} .$$

\end{definition}

\begin{example} For $n=5$ and $t=2$ the simplicial complex $\Omega^5_2$ has facets $\{\hat{1},\hat{2},3,4,5\}$, $\{1,\hat{2},\hat{3},4,5 \}$, $\{1,2,\hat{3},\hat{4},5 \}$ and, $\{1,2,3,\hat{4},\hat{5} \} $.
\end{example}
\begin{remark} For $n=t$ the simplicial complex $\Omega_t^n$ is the irrelevant complex $\{\emptyset \}$. If $t=1$ then $\Omega^n_1$ coincides with the boundary of an $n-1$ dimensional simplex.
\end{remark}
As the simplicial complex $\Omega^n_t$ will come up in the next sections, we study its homology groups.
\begin{lemma}\label{recursive formula for omega_n,t} For $n\geq 2t+1$ we have $\widetilde{H}_p(\Omega_t^n)\cong \widetilde{H}_{p-2}(\Omega_t^{n-t-1})$. Otherwise,
\begin{equation}
	\widetilde{H}_{p}(\Omega_t^n) \cong
    \begin{cases}
        \widetilde{H}_{p}(\{\emptyset \}) & \text{ if } n=t \\
      \widetilde{H}_{p-1}(\{\emptyset \}) & \text{ if } n=t+1 \\
      0 & \text{ if } t+2\leq n \leq 2t.
      
    \end{cases}
  \end{equation}
\end {lemma}

\begin{proof} The case $n=t$ is clear as $\Omega_t^t=\{\emptyset\}$. So we assume that $n >t$ and fix an index $p$. We write $\Omega_t^n=S\cup C$ where $S$ is the simplex on vertices $\{t+1,...,n\}$ and C is the cone generated by the facets of $\Omega_t^n$ that contain the vertex $1$. Note that by Corollary \ref{Corollary to MV acyclic union} we have $$\widetilde{H}_p(\Omega_t^n)\cong \widetilde{H}_{p-1}(S\cap C).$$ We consider the three cases left:

Case 1: If $n=t+1$ then $S\cap C$ is the irrelevant complex and we are done.

Case 2: If $t+2 \leq n \leq 2t$ then $S\cap C$ is a simplex whose maximal face is $\{t+2,...,n\}$. 

Case 3: If $n\geq 2t+1$ then it is not hard to check that $S\cap C$ can be written as a union $S \cap C= S_1 \cup C_1$ where $S_1=\langle \{t+2,...,n\}\rangle$ and, $C_1$ is the cone with apex $t+1$ such that
$$\operatorname{Facets}(C_1)=\big\{\{t+1,...,n\}-\{i,i+1,...,i+t-1\} \mid i=t+2,...,n-t+1\big\}. $$
Now observe that $S_1\cap C_1\cong \Omega_t^{n-t-1}$ and again by Corollary \ref{Corollary to MV acyclic union} we get $\widetilde{H}_{p-1}(S\cap C)\cong \widetilde{H}_{p-2}(S_1\cap C_1)\cong \widetilde{H}_{p-2}(\Omega_t^{n-t-1})$ which completes the proof.
\end{proof}

\begin{theorem}\label{homology of omega} The homology groups of $\Omega_t^n$ are given by
\begin{equation}\label{homology of omega equation}
	\widetilde{H}_{p}(\Omega_t^n) \cong
    \begin{cases}
       \widetilde{H}_{p+1-\frac{2n}{t+1}}(\{\emptyset\}) & \text{ if } n\equiv 0 \bmod{t+1} \\
      \widetilde{H}_{p+2-\frac{2(n+1)}{t+1}}(\{\emptyset \}) & \text{ if } n\equiv t  \bmod{t+1}  \\
      0 & \text{ otherwise. } 
      
    \end{cases}
  \end{equation}

\end{theorem}
\begin{proof} 
Follows by a straightforward induction using Lemma \ref{recursive formula for omega_n,t}.
\end{proof}
\begin{corollary}\label{dimension of homology of omega}The dimensions of reduced homologies of $\Omega_t^n$ are independent of the ground field. And they are given by
\begin{equation} 
	\dim \widetilde{H}_{p}(\Omega_t^n) =
    \begin{cases}
       \delta_{p+2,\frac{2n}{t+1}} & \text{ if } n\equiv 0  \bmod{t+1}  \\
      \delta_{p+3,\frac{2(n+1)}{t+1}} & \text{ if } n\equiv t  \bmod{t+1}  \\
      0 & \text{ otherwise. } 
      \end{cases}
  \end{equation}
\end{corollary}
\begin{proof}Follows by Theorem \ref{homology of omega} and Equation \eqref{homology of irrelevant complex}.
\end{proof}
\subsection{Lines and Cyles}
Throughout this section let $\Delta$ be the Taylor simplex of $I_t(L_n)$ where $L_n$ is a line on vertices $x_1,...,x_n$. If $n < t$ then there is no path on $L_n$ of length $t-1$, and therefore $I_t(L_n)=0$. Let us assume $n\geq t$ then we have
$$ \Delta= \langle x_ix_{i+1}...x_{i+t-1} \mid i=1,...,n-t+1 \rangle .$$ 
For simplicity, we replace the label of a vertex $x_ix_{i+1}...x_{i+t-1}$ with $i$ for all $i=1,...,n-t+1$. Hence $\Delta$ can be viewed as a simplex with maximal face $\{1,2,...,n-t+1\}$. Now we want to find $\Delta_{<m}$. Following the Equation \eqref{facets of theta less than m}, the maximal elements of
\begin{gather*}
\big\{\{\hat{1},2,...,n-t+1\}, \{1,...,n-t,\hat{n-t+1}\}\big\}  \\
\cup \big\{\{\hat{1},...,\hat{i},i+1,...,n-t+1\} \mid i=2,...,t \big\}\\
\cup \big\{ \{1,...,i-1,\hat{i},\widehat{i+1},...,\widehat{i+t-1},i+t,...,n-t+1 \} \mid i=2,...,n-2t+2\big\}\\
\cup \big\{ \{1,...,i-1,\hat{i},...,\widehat{n-t+1} \} \mid i=n-2t+3,...,n-t\big\}
\end{gather*}
give the facets of $\Delta_{<m}.$ Therefore, if $n<2t+1$ then
\begin{equation}\label{n<2t+1 line equation}
\Delta_{<m} = \langle \{\hat{1},2,...,n-t+1\}, \{1,...,n-t,\widehat{n-t+1}\} \rangle.
\end{equation}
And, if $n\geq 2t+1$ we have the following equation.
\begin{equation}\label{n >= 2t+1 line equation}
\begin{split}
\operatorname{Facets}(\Delta_{<m}) & = \big\{\{\hat{1},2,...,n-t+1\}, \{1,...,n-t,\widehat{n-t+1}\}  \big\}  \\ & \ \  \cup \big\{\{1,...,i-1,\hat{i},\widehat{i+1},...,\widehat{i+t-1},i+t,...,n-t+1\}  \mid i=2,...,n-2t+1 \big\}
\end{split}
\end{equation}
\begin{theorem}[\textbf{Top grade Betti numbers of path ideals of lines}]\label{top grade Betti numbers of path ideals of lines}For all $i\geq 1$, and $n\geq 1$, we have
\begin{equation}\label{equation in top grade Betti numbers of path ideals of lines } 
	b_{i,n}(S/I_t(L_n)) =
    \begin{cases}
       \delta_{i,\frac{2n}{t+1}} & \text{ if } n\equiv 0  \bmod{t+1}  \\
      \delta_{i+1,\frac{2n+2}{t+1}} & \text{ if } n\equiv t  \bmod{t+1}  \\
      0 & \text{ otherwise. } 
      \end{cases}
  \end{equation}
 \end{theorem}
\begin{proof}
First suppose that $n <t$, then we know that $I_t(L_n)=0$. As $n$ cannot be $0$ or $t \bmod{t+1}$ in this case we are done.

Now we assume that $n\geq t$. Let $m$ be the product of vertices of $L_n$. By Equation \eqref{definition of graded betti number} and Theorem \ref{betti number by homology} we have
$$ b_{i,n}(S/I_t(L_n))=b_{i,m}(S/I_t(L_n))=\operatorname{dim}_{\Bbbk}\widetilde{H}_{i-2}(\Delta_{<m},\Bbbk).$$
We consider two cases:

Case 1: Suppose that $n<2t+1$. By Equation \eqref{n<2t+1 line equation} we have
$$\Delta_{<m} = \langle \{\hat{1},2,...,n-t+1\}, \{1,...,n-t,\widehat{n-t+1}\} \rangle. $$ Then we have three cases to prove. If $n=t$, then $\Delta_{<m}=\{\emptyset\}$ and so that
$$ \operatorname{dim}_{\Bbbk}\widetilde{H}_{i-2}(\Delta_{<m},\Bbbk)= \operatorname{dim}_{\Bbbk}\widetilde{H}_{i-2}(\{\emptyset\},\Bbbk)=\delta_{i-2,-1}. $$
Observe that $\delta_{i-2,-1}=\delta_{i+1,\frac{2n+2}{t+1}}$ for $n=t$ which proves Equation \eqref{equation in top grade Betti numbers of path ideals of lines } for this case. Now observe that if $n>t$ then $\Delta_{<m}$ is a union of two simplices
$$\Delta_{<m} = \langle \{\hat{1},2,...,n-t+1\}\rangle \cup \langle \{1,...,n-t,\widehat{n-t+1}\} \rangle =S_1\cup S_2. $$
Hence by Corollary \ref{Corollary to MV acyclic union}, $\operatorname{dim}_{\Bbbk}\widetilde{H}_{i-2}(\Delta_{<m},\Bbbk)\cong \operatorname{dim}_{\Bbbk}\widetilde{H}_{i-3}(S_1\cap S_2,\Bbbk)$. If  $n=t+1$ then $S_1\cap S_2$ is the irrelevant complex. Therefore we have
$$\operatorname{dim}_{\Bbbk}\widetilde{H}_{i-3}(S_1\cap S_2,\Bbbk)\cong \operatorname{dim}_{\Bbbk}\widetilde{H}_{i-3}(\{\emptyset\},\Bbbk)=\delta_{i-3,-1}. $$
Now we check that indeed $\delta_{i-3,-1}=\delta_{i,\frac{2n}{t+1}}$ for $n=t+1$ and the proof follows for this case. Next, if $n\geq t+1$ then $S_1\cap S_2$ is a simplex and has trivial reduced homology in all degrees.

Case 2: Suppose that $n\geq 2t+1$. Then by Equation \eqref{n >= 2t+1 line equation} we have
$\Delta_{<m} = S_1  \cup S_2 \cup \Upsilon$ where 
$\Upsilon= \langle\{1,...,i-1,\hat{i},\widehat{i+1},...,\widehat{i+t-1},i+t,...,n-t+1\}  \mid i=2,...,n-2t+1 \rangle $, $S_1=\langle \{\hat{1},2,...,n-t+1\}\rangle$ and $S_2=\langle \{1,...,n-t,\widehat{n-t+1}\} \rangle. $ Now we write $\Delta_{<m}$ as a union $\Delta_{<m}=S_1\cup(S_2\cup \Upsilon)$ where $S_2\cup \Upsilon$ is a cone with apex $1$. By virtue of Corollary \ref{Corollary to MV acyclic union} we have 
$$\operatorname{dim}_{\Bbbk}\widetilde{H}_{i-2}(\Delta_{<m},\Bbbk)\cong \operatorname{dim}_{\Bbbk}\widetilde{H}_{i-3}(S_1 \cap (S_2\cup \Upsilon),\Bbbk).  $$
Now observe that $S_1\cap(S_2\cup \Upsilon)= C \cup S_2$ where $C$ is the cone generated by the facets of $S_1 \cap (S_2\cup \Upsilon)$ that contain the vertex $n-t+1$. Again by Corollary \ref{Corollary to MV acyclic union} we get
$$\operatorname{dim}_{\Bbbk}\widetilde{H}_{i-3}(S_1 \cap (S_2\cup \Upsilon),\Bbbk) \cong \operatorname{dim}_{\Bbbk}\widetilde{H}_{i-4}(C\cap S_2, \Bbbk). $$
Note that $C\cap S_2$ is isomorphic to the simplicial complex $\Omega_{t}^{n-t-1}$ and by Corollary \ref{dimension of homology of omega} we have.
\begin{equation} 
	\dim \widetilde{H}_{i-4}(\Omega_t^{n-t-1}) =
    \begin{cases}
       \delta_{i-2,\frac{2(n-t-1)}{t+1}} & \text{ if } n\equiv 0  \bmod{t+1}  \\
      \delta_{i-1,\frac{2(n-t)}{t+1}} & \text{ if } n\equiv t  \bmod{t+1}  \\
      0 & \text{ otherwise } 
      \end{cases}
  \end{equation}
  which agrees with the formula given in Equation \eqref{equation in top grade Betti numbers of path ideals of lines }.
\end{proof}
\begin{theorem}[\textbf{Multigraded Betti numbers of path ideals of lines}]\label{multigraded Betti numbers of path ideals of lines} Let $t\geq 2$ and $m$ be a squarefree monomial of degree $j$. Then the multigraded Betti number $b_{i,m}(S/I_t(L_n))=1$ if the induced graph $(L_n)_m$ consists of a collection of disjoint lines that satisfy the following conditions:
\begin{itemize} 
\setlength{\itemsep}{0pt}
  \item[\textnormal{(i)}] Each line is of order $0$ or $t\bmod{t+1}$
  \item[\textnormal{(ii)}] The number of lines of order $t \bmod{t+1}$ is equal to $\frac{i(t+1)-2j}{1-t}$.
 \end{itemize}
Otherwise, $b_{i,m}(S/I_t(L_n))=0$.
\end{theorem}
\begin{proof}
Let $(L_n)_m= \cup_{l=1}^{p}Q_i$ be a disjoint union of lines where each $Q_l$ is a line of order $v_l$. We have
\begin{equation*}
\begin{split}
 b_{i,m} (S/I_t(L_n)) & =   b_{i,j} (S/I_t((L_n)_m))  \text{ by Lemma \ref{multigraded betti number and induced subgraph}} \\ & =   \sum_{u_1+...+u_p=i } b_{u_1,v_1}(S/I_t(Q_1)).... b_{u_p,v_p}(S/I_t(Q_p))  \text{ by Equation \eqref{top betti numbers of connected components}.}
\end{split}
\end{equation*}
By Theorem \ref{top grade Betti numbers of path ideals of lines} if one of $Q_l$ is not of order $0$ or $t \bmod{t+1}$ then the sum above is zero. So without loss of generality let us assume that $Q_1,...,Q_z$ are of order  $0 \bmod{t+1}$ and $Q_{z+1},...,Q_{p}$ are of order $t \bmod{t+1}$ for some $0 \leq z \leq p$.  Again by Theorem \ref{top grade Betti numbers of path ideals of lines}, the sum above is equal to $1$ if
\begin{equation}\label{equation in proof of multigraded betti numbers of paths}
\sum_{l=1}^z\frac{2v_l}{t+1} + \sum_{l=z+1}^p(\frac{2v_l+2}{t+1}-1) = i
\end{equation}
and zero otherwise. Observe that \eqref{equation in proof of multigraded betti numbers of paths} holds iff $p-z =\frac{i(t+1)-2j}{1-t}$ since $v_1+...+v_p=j$.
Hence the result follows.
\end{proof}
\begin{corollary}\label{graded betti number of lines combinatorial}If $L$ is a line, $b_{i,j} (S/I_t(L))$ is the number of ways of choosing a collection of disjoint induced lines of $L$ that satisfy the following conditions:
\begin{itemize}
\setlength{\itemsep}{0pt}
\item[\textnormal{(i)}]The orders of the lines add up to $j$
\item[\textnormal{(ii)}] Each line is of order $0$ or $t\bmod{t+1}$
\item[\textnormal{(iii)}]The number of lines of order $t \bmod{t+1}$ is equal to $\frac{i(t+1)-2j}{1-t}$.
\end{itemize}
\end{corollary}
\begin{proof}Immediately follows by Theorem \ref{multigraded Betti numbers of path ideals of lines} and Equation \eqref{definition of graded betti number}.
\end{proof}
Using the multigraded Betti numbers, we can calculate graded Betti numbers. The following result was also proved in \cite{alilooee faridi}.
\begin{theorem}[\textbf{Graded Betti numbers of path ideals of lines}]\label{graded betti numbers of paths exact formula} For $t \geq 2$, the nonzero graded Betti numbers of $S/I_t(L_n)$ are given by
 $$b_{i,j}(S/I_t(L_n))= {n-j+1 \choose \frac{i(t+1)-2j}{1-t} }{ n-j+\frac{j-ti}{1-t} \choose n-j  }$$
provided that $n, i$ and $j$ satisfy the following relations.
\begin{itemize}
\setlength{\itemsep}{0pt}
\item[\textnormal{(i)}]$n \geq j$
\item[\textnormal{(ii)}] $j \geq t\big(\frac{i(t+1)-2j}{1-t}\big)\geq 0$
\item[\textnormal{(iii)}]$n-j\geq \frac{i(t+1)-2j}{1-t}-1$
\end{itemize}
Otherwise, the  graded Betti numbers are zero. 
\end{theorem}
\begin{proof}By Lemma \ref{graded betti number of lines combinatorial} it is clear that $b_{i,j}(S/I_t(L_n))=0$ if the condition $(i)$ or $(ii)$ fails. So let us assume that $(i)$ and $(ii)$ hold. 

Now suppose that we have chosen a collection of disjoint induced lines $Q_1,...,Q_p$ of $L_n$ as in Lemma \ref{graded betti number of lines combinatorial}. Since the orders of $Q_1,...,Q_p$ add up to $j$, we have $j=\lvert\cup_{k=1}^pV(Q_k)\rvert$. Also as the number of lines of order $t \bmod{t+1}$ is equal to $\frac{i(t+1)-2j}{1-t}$, at least $t\big(\frac{i(t+1)-2j}{1-t}\big)$ vertices of $\cup_{k=1}^pV(Q_k)$ belong to a line of order $t \bmod{t+1}$. Therefore at most $j-t\big(\frac{i(t+1)-2j}{1-t}\big)=(1+t)(\frac{j-ti}{1-t})$ vertices of $\cup_{k=1}^pV(Q_k)$ belong to a line of order $0 \bmod{t+1}$. Now it becomes clear that the problem of choosing a collection of disjoint induced lines of $L_n$ that is described in Lemma \ref{graded betti number of lines combinatorial} corresponds to the problem of ordering $\frac{i(t+1)-2j}{1-t}$ many $``t"$s, $\frac{j-ti}{1-t}$ many $``1+t"$s and $``n-j"$ many points on a row such that there is a point between any $``t"$s and the order of $``t"$s and $``t+1"$s between two points is ignored. (Note that for example, in the latter interpretation the orderings
$$\cdot \ t \ (t+1) \ (t+1) \ \cdot \ \cdot \ (t+1) \ t \ \cdot \ t \  \text{  and  } \  \cdot \ (t+1) \ t \ (t+1) \ \cdot \ \cdot \ t \ (t+1) \ \cdot \ t  $$
are considered as the same since they both correspond to the collection $L_{t+2(t+1)}, L_{(t+1)+t}, L_t$ where $$L_n=L_{3(t+1)+3t+4}=L_1 \cup L_{t+2(t+1)} \cup L_1 \cup L_1 \cup L_{(t+1)+t} \cup L_1 \cup L_t.)$$

Now to count the number of solutions to this problem we spread $\frac{i(t+1)-2j}{1-t}$ many $``t"$s on a row and put one point between any two: 
$$ t \ \cdot \  t \ \cdot \ t \ \cdot   \ ...\  \cdot \ t $$  

Observe that to achieve this there must be at least $\frac{i(t+1)-2j}{1-t}-1$ many points, i.e. $n-j\geq \frac{i(t+1)-2j}{1-t}-1 $ which is condition $(iii)$.

Now we are allowed to insert the remaining $n-j-(\frac{i(t+1)-2j}{1-t}-1)$ points. Observe that we have $\frac{i(t+1)-2j}{1-t}+1$ many places to put each of them as indicated with $ - $ below.
$$ - \ t \ - \ \cdot   \  t \ - \ \cdot  \ t \ - \ \cdot   \ ...\ - \  \cdot \  t \ -$$  
This is equivalent to finding the number of integer solutions to the equation
$$A_1+A_2+\ ... + \ A_{\frac{i(t+1)-2j}{1-t}+1}= n-j-\left(\frac{i(t+1)-2j}{1-t}-1 \right)$$
with $A_i \geq 0$, which is $ {n-j+1  \choose \frac{i(t+1)-2j}{1-t} }. $ 

Finally we insert the $``t+1"$s. Since the order of $``t"$s and $``t+1"$s between two points is ignored there are $n-j+1$ places (spaces between two points plus endpoints) to insert each $t+1$. 
But the number of ways of doing this is equal to the number of integer solutions of the equation
$$A_1+A_2+\ ... + \ A_{n-j+1}= \frac{j-ti}{1-t}$$
with $A_i \geq 0$, which is ${n-j+\frac{j-ti}{1-t} \choose n-j}. $ Hence the number of all possible collections is equal to
$$ {n-j+1  \choose \frac{i(t+1)-2j}{1-t} }{n-j+\frac{j-ti}{1-t} \choose n-j}  $$ and the proof is completed.
\end{proof}
\begin{corollary}[\textbf{Multigraded Betti numbers of path ideals of cycles}]\label{multigraded Betti numbers of cycles}Let $t\geq 2$ and $m$ be a squarefree monomial of degree $j<n$. Then the multigraded Betti number $b_{i,m}(S/I_t(C_n))=1$ if the induced graph $(C_n)_m$ consists of a collection of disjoint lines that satisfy the following conditions:
\begin{itemize}
\setlength{\itemsep}{0pt}
  \item[\textnormal{(i)}] Each line is of order $0$ or $t\bmod{t+1}$
  \item[\textnormal{(ii)}] The number of lines of order $t \bmod{t+1}$ is equal to $\frac{i(t+1)-2j}{1-t}$.
 \end{itemize}
Otherwise, $b_{i,m}(S/I_t(L_n))=0$.
\end{corollary}
\begin{proof}By Lemma \ref{multigraded betti number and induced subgraph} we have $b_{i,m}(S/I_t(C_n))=b_{i,j}(S/I_t((C_n)_m))$. Since $(C_n)_m$ is a disjoint union of lines the proof follows by Theorem \ref{multigraded Betti numbers of path ideals of lines}.
\end{proof}



In the next Corollary, we will give a formula for the graded Betti numbers $b_{i,j}(S/I_t(C_n))$ when $j<n$. Note that Theorem 4.13 of \cite{alilooee faridi 2} gives a complete formula for all Betti numbers.
\begin{corollary}[\textbf{Graded Betti numbers of path ideals of cycles}]\label{graded betti numbers of cycles exact formula}For $j<n$ and $t\geq 2$ the graded Betti numbers of $S/I_t(C_n)$ are given by
$$b_{i,j}(S/I_t(C_n))=\frac{n}{n-j}{n-j \choose \frac{i(t+1)-2j}{1-t}}{n-j-1+\frac{j-ti}{1-t} \choose n-j-1} $$
provided that
\begin{itemize}
\setlength{\itemsep}{0pt}
\item[\textnormal{(i)}]$n-1 \geq j$
\item[\textnormal{(ii)}]$j \geq t\big(\frac{i(t+1)-2j}{1-t}\big)\geq 0$
\item[\textnormal{(iii)}]$n-j \geq \frac{i(t+1)-2j}{1-t}$.
\end{itemize}
Otherwise, the graded Betti numbers are zero.
\end{corollary}
\begin{proof}By Corollary \ref{multigraded Betti numbers of cycles} and Equation \eqref{definition of graded betti number}, $b_{i,j}(S/I_t(C_n))$ is the number of ways one can choose a collection of disjoint induced lines on $C_n$ such that the orders of the lines add up to $j$, each line is of order $0$ or $t \bmod{t+1}$ and, the number of lines of order $t \bmod{t+1}$ is equal to $\frac{i(t+1)-2j}{1-t}$. Then this is a problem of ordering $\frac{i(t+1)-2j}{1-t}$ many $``t"$s, $\frac{j-ti}{1-t}$ many $``t+1"$s and $n-j$ many points around a circle such that there is at least one point between any $``t"$s and the order of $``t"$s and $``t+1"$s between two points is ignored. Any such ordering can be obtained by first fixing a point on the cycle and ordering the remaining $n-j-1$ points, $\frac{i(t+1)-2j}{1-t}$ many $``t"$s, $\frac{j-ti}{1-t}$ many $``t+1"$s on a row with the same conditions. By Theorem \ref{graded betti numbers of paths exact formula}, there are ${n-j \choose \frac{i(t+1)-2j}{1-t}}{n-j-1+\frac{j-ti}{1-t} \choose n-j-1}$ ways to do it. Also there are $n$ choices to fix a vertex on the cycle. However it is clear that $n{n-j \choose \frac{i(t+1)-2j}{1-t}}{n-j-1+\frac{j-ti}{1-t} \choose n-j-1}$ will give an overcount since fixing different points may yield the same ordering. To overcome this problem, consider a circle with a desired ordering. It has $n-j$ points and this ordering was counted once for fixing each of these points. Hence the result follows.
\end{proof}


\subsection{Stars}
Throughout this section $\mathcal{S}_n$ will be a star graph of size $n$.

\begin{lemma}\label{boundary of Delta} Let $G$ be a connected graph, let $\Delta$ be the Taylor simplex of $I_2(G)$. If $m$ is the product of the vertices of $G$ then the simplicial complex $\Delta_{< m}$ is the boundary of $\Delta$ iff $G$ is a star.
\end{lemma}

\begin{proof}Suppose that $e_1,...,e_q$ are the edges of the graph $G$. Then, 
$\Delta_{< m}$ is the boundary of $\Delta$ iff
$F_i=\{e_1,...,e_q\} - \{e_i\}$ is a facet of  $\Delta_{<m}$ for each $i=1,...,q.$ The latter holds only if multidegree of $F_i$ properly divides $m$ for every $i$. Or, equivalently
each $e_i$ contains a vertex $x_i$ such that $x_i \notin \cup_{j \neq i}e_j $. But this happens only if $G$ is a star since $G$ is connected.
\end{proof}

\begin{corollary}\label{top betti of star graphs}
Let $G$ be a star on $d+1$ vertices. Then
\begin{equation*}
     b_{i,d+1} (S/I(G))=
    \begin{cases}
      1 & \text{if } i=d \\
      0 & \text{otherwise}
    \end{cases}
  \end{equation*}
\end{corollary}
\begin{proof} Follows by combining Lemma \ref{boundary of Delta}, Theorem \ref{betti number by homology} and Equation \eqref{homology of boundary}.
\end{proof}
\begin{corollary}\label{graded betti numbers of edge ideals of stars}Let $G$ be a star on $d+1$ vertices. Then the graded Betti numbers of $I(G)$ are given by 
\begin{equation*}
     b_{i,d+1-j} (S/I(G))=
    \begin{cases}
      {d \choose j} & i=d-j \\
      0 & \text{otherwise}
    \end{cases}
  \end{equation*}
\end{corollary}
\begin{proof}Fix $j$ and recall Equation \eqref{definition of graded betti number} and Lemma \ref{multigraded betti number and induced subgraph}. Any induced subgraph of $G$ is either a star or contains isolated vertex. If it contains an isolated vertex then by Remark \ref{isolated vertex} the multigraded Betti number for such induced subgraph is zero. Hence by Corollary \ref{top betti of star graphs} we see that $b_{i,d+1-j} (S/I(G))$ is the number of induced star subgraphs of $G$ of order $d+1-j$ if $i=d-j$ and zero otherwise.
\end{proof}

\begin{proposition}\label{proposition simplicial complex}
Let $\Gamma$ be a simplicial complex which is not a cone. Suppose that $\langle F_1,...,F_q \rangle = \Gamma$ and there exists a sequence of distinct vertices $v_1,...,v_q$ of $\Gamma$ such that $v_i \notin F_j $ iff $i=j$. Then $\widetilde{H}_p(\Gamma, \Bbbk) \cong \widetilde{H}_{p-q+1}(\{\emptyset \}, \Bbbk) $ for any field $\Bbbk$.
\end{proposition}

\begin{proof}
We induct on $q$, the number of facets. Since there is no simplex which satisfies the assumptions of the given Proposition, the basis step starts at $q=2$. 

Suppose that $\Gamma=\langle F_1, F_2 \rangle$ is not a cone and it has two vertices $v_1, v_2$ such that $v_i \notin F_j \Leftrightarrow i=j$. Then $\langle F_1 \rangle \cap \langle F_2 \rangle \cong \{\emptyset \}$. Since $\Gamma= \langle F_1 \rangle \cup \langle F_2 \rangle$ and $\langle F_1 \rangle, \langle F_2 \rangle$ are acyclic, by virtue of Corollary \ref{Corollary to MV acyclic union} we have $\widetilde{H}_p(\Gamma) \cong \widetilde{H}_{p-1}(\langle F_1 \rangle \cap \langle F_2 \rangle)= \widetilde{H}_{p-1}(\{\emptyset\})$ as desired.

Now let $\Gamma= \langle F_1,..., F_q \rangle, q\geq 3$ be a simplicial complex as in the statement of the Proposition. We write 
$$\Gamma = \langle F_1,...,F_{q-1} \rangle \cup \langle F_q \rangle $$ 
where $\langle F_q \rangle$ is a simplex and $\langle F_1,...,F_{q-1} \rangle$ is a cone with apex $v_q$. By Corollary \ref{Corollary to MV acyclic union} we have $\widetilde{H}_p(\Gamma)\cong \widetilde{H}_{p-1}(\langle F_1,...,F_{q-1} \rangle \cap \langle F_q \rangle)$. But observe that
$$\langle F_1,...,F_{q-1} \rangle \cap \langle F_q \rangle = \langle F_1\cap F_q,...,F_{q-1}\cap F_q \rangle $$
as $v_j \in F_i\cap F_q, v_j \notin F_j \cap F_q$ so that $F_i\cap F_q \nsubseteq F_j \cap F_q$ for all $1 \leq i \neq j \leq q-1$. Clearly, the simplicial complex $\langle F_1\cap F_q,...,F_{q-1}\cap F_q \rangle$ is not a cone and moreover
$$v_i \notin F_j \cap F_q \Leftrightarrow i=j$$  
for all  $1\leq i,j \leq q-1$. Hence it satisfies the inductive hypothesis and we get 
$$\widetilde{H}_{p-1}(\langle F_1\cap F_q,...,F_{q-1}\cap F_q \rangle)\cong \widetilde{H}_{p-q+1}(\{\emptyset\}) $$
which completes the proof.
\end{proof}
\begin{theorem}\label{top betti numbers of stars for path ideals of length two}Let $\mathcal{S}_n$ be a star graph of size $n\geq 2$. Then for all $i \geq 1$
\begin{equation}
b_{i,n+1}(S/I_3(\mathcal{S}_n)) =
    \begin{cases}
       i & \text{ if } n+1 = i+2  \\
      0 & \text{ otherwise. } 
      \end{cases}
\end{equation}
\end{theorem}
\begin{proof} Let $\mathcal{S}_n$ be a star graph of size $n$ with the edge set $E(\mathcal{S}_n)=\{\{x_0, x_i\} \mid i=1,...,n\}$. Suppose that $\Delta(\mathcal{S}_n)$ is the Taylor simplex of $I_3(\mathcal{S}_n)$. We prove the given statement by induction on $n$ and using Theorem \ref{betti number by homology} so that for all $i\geq 1$
$$b_{i,n+1}(S/I_3(\mathcal{S}_n)) = \operatorname{dim}_{\Bbbk}\widetilde{H}_{i-2}(\Delta(\mathcal{S}_n)_{<x_0...x_n}, \Bbbk). $$
For $n=2$, we have $\Delta(\mathcal{S}_2)_{<x_0x_1x_2} \cong \{\emptyset\}$ so the basis step is settled by \eqref{homology of irrelevant complex}.

Next we consider $\Delta(\mathcal{S}_n)_{<x_0x_1...x_n}$ for some $n\geq 3$. We have a decomposition
\begin{equation}\label{union star path ideal}\Delta(\mathcal{S}_n)_{<x_0x_1...x_n}=\Delta(\mathcal{S}_n)_{\leq x_0x_1...x_{n-1}} \bigcup (\bigcup_{i=1}^{n-1}\Delta(\mathcal{S}_n)_{\leq \frac{x_0x_1...x_n}{x_i}}) 
\end{equation} 
since $\Delta(\mathcal{S}_n)_{\leq x_1...x_n}$ is isomorphic to the irrelevant complex. For $i\geq 1$ we set
\begin{equation}\label{definition of F_i in stars}
\Delta(\mathcal{S}_n)_{\leq \frac{x_0...x_n}{x_i}}=\langle F_i \rangle := \langle \{x_0x_jx_k \mid j,k\in\{1,...,n\}-\{i\} \text{ and } j<k\}\rangle
\end{equation}
and note that every element of $\{F_1,...,F_n\}$ is maximal with respect to inclusion because of the symmetry of star graphs. By Equation \eqref{facets of theta less than m} we get $\Delta(\mathcal{S}_n)_{<x_0...x_n}=\langle F_1,...,F_n \rangle$. Observe that \eqref{union star path ideal} becomes
\begin{equation}\label{decomposition stars}
\Delta(\mathcal{S}_n)_{<x_0...x_n}= \langle F_n\rangle \cup \langle F_1,...,F_{n-1} \rangle
\end{equation}
by definition of $F_i$. Now we claim the followings:
\begin{itemize}
  \setlength{\itemsep}{0pt}
\item[(i)] $\langle F_n \rangle \cap \langle F_1,...,F_{n-1}\rangle \cong \Delta(\mathcal{S}_{n-1})_{<x_0...x_{n-1}}$
\item[(ii)]$\widetilde{H}_p(\langle F_1,...,F_{n-1} \rangle) \cong \widetilde{H}_{p-n+2}(\{\emptyset\})$.
\end{itemize}

Claim (i) is trivial as $F$ is a facet of $\langle F_n \rangle \cap \langle F_1,...,F_{n-1}\rangle$ iff $F=F_n\cap F_i$ for some $1\leq i \leq n-1$. But the latter means that $F$ consists of all paths of the form $x_0x_jx_k$ where $j,k\in\{x_1,...,x_n\}-\{x_i,x_n\}$ and $j \neq k$.

For claim (ii) we show that Proposition \ref{proposition simplicial complex} applies to the simplicial complex $\langle F_1,...,F_{n-1} \rangle$.To this end, we first check that $\langle F_1,...,F_{n-1} \rangle$ is not a cone. Assume for a contradiction it is a cone with apex $x_0x_ix_j$. Then $x_0x_ix_j \in F_1\cap...\cap F_{n-1}$ and $i,j\in \{1,...,n\}-\{1,...,n-1\}$. Thus $i=j=n$ which is a contradiction. Now let $v_1=x_0x_1x_n,...,v_{n-1}=x_0x_{n-1}x_n$ be a sequence of vertices of $\langle F_1,...,F_{n-1} \rangle$. Clearly we have $v_i\notin F_j \Leftrightarrow i=j$ which proves claim (ii). Therefore \eqref{homology of irrelevant complex} yields
\begin{equation}\label{eqn for mv seq in proof of stars}
\operatorname{dim}\widetilde{H}_p(\langle F_1,...,F_{n-1} \rangle) =
    \begin{cases}
       1 & \text{ if } p=n-3  \\
      0 & \text{ otherwise. } 
      \end{cases}
\end{equation}
Also by inductive assumption and (i) we have
\begin{equation}\label{eqn for mv intersection in proof of stars}
\operatorname{dim}\widetilde{H}_p(\langle F_n\rangle \cap \langle F_1,...,F_{n-1} \rangle) =
    \begin{cases}
       p+2 & \text{ if } p=n-4  \\
      0 & \text{ otherwise. } 
      \end{cases}
\end{equation}
Therefore Mayer-Vietoris sequence for \eqref{decomposition stars} is
\begin{gather*}
 .... \rightarrow  0 \rightarrow 0 \rightarrow \widetilde{H}_{n-1}(\Delta(\mathcal{S}_n)_{<x_0...x_n}) \rightarrow 0 \rightarrow 0 \rightarrow \widetilde{H}_{n-2}(\Delta(\mathcal{S}_n)_{<x_0...x_n}) \rightarrow 0 \rightarrow \\
   \widetilde{H}_{n-3}(\langle F_1,...,F_{n-1} \rangle) \rightarrow \widetilde{H}_{n-3}(\Delta(\mathcal{S}_n)_{<x_0...x_n}) \rightarrow \widetilde{H}_{n-4}(\langle F_n\rangle \cap \langle F_1,...,F_{n-1}\rangle) \rightarrow \\
  0 \rightarrow \widetilde{H}_{n-4}(\Delta(\mathcal{S}_n)_{<x_0...x_n}) \rightarrow 0 \rightarrow 0 \rightarrow \widetilde{H}_{n-5}(\Delta(\mathcal{S}_n)_{<x_0...x_n}) \rightarrow 0  \rightarrow 0 \rightarrow ....
    \end{gather*}
For $i\leq n-5$ and $i\geq n-2$ we have the sequence
$$0 \rightarrow 0 \rightarrow \widetilde{H}_{i}(\Delta(\mathcal{S}_n)_{<x_0...x_n}) \rightarrow 0 $$
which implies that $\widetilde{H}_{i}(\Delta(\mathcal{S}_n)_{<x_0...x_n})=0$. Hence Mayer-Vietoris sequence above becomes
\begin{gather}\label{mv seq in star}... \rightarrow 0 \rightarrow 0 \rightarrow 0 \rightarrow  \widetilde{H}_{n-3}(\langle F_1,...,F_{n-1} \rangle) \rightarrow  \widetilde{H}_{n-3}(\Delta(\mathcal{S}_n)_{<x_0...x_n}) \rightarrow \\ \label{mv seq in star line two}
\widetilde{H}_{n-4}(\langle F_n\rangle \cap \langle F_1,...,F_{n-1}\rangle) \rightarrow 0 \rightarrow \widetilde{H}_{n-4}(\Delta(\mathcal{S}_n)_{<x_0...x_n}) \rightarrow 0 \rightarrow 0 \rightarrow 0 \rightarrow ... 
\end{gather}
and by \eqref{mv seq in star line two} we see that $\widetilde{H}_{n-4}(\Delta(\mathcal{S}_n)_{<x_0...x_n})=0$. Hence \eqref{mv seq in star} and \eqref{mv seq in star line two} turn into
\small
$$...\rightarrow 0 \rightarrow \widetilde{H}_{n-3}(\langle F_1,...,F_{n-1} \rangle) \rightarrow  \widetilde{H}_{n-3}(\Delta(\mathcal{S}_n)_{<x_0...x_n}) \rightarrow \widetilde{H}_{n-4}(\langle F_n\rangle \cap \langle F_1,...,F_{n-1}\rangle) \rightarrow 0 \rightarrow ... $$
\normalsize 
which gives that 
\begin{equation*}
\begin{split}
 \operatorname{dim}\widetilde{H}_{n-3}(\Delta(\mathcal{S}_n)_{<x_0...x_n})  & = \operatorname{dim}\widetilde{H}_{n-3}(\langle F_1,...,F_{n-1} \rangle) + \operatorname{dim}\widetilde{H}_{n-4}(\langle F_n\rangle \cap \langle F_1,...,F_{n-1}\rangle)  \\ & =  1 + (n-2) \text{ by } \eqref{eqn for mv seq in proof of stars} \text{ and } \eqref{eqn for mv intersection in proof of stars} \\ & = n-1
\end{split}
\end{equation*}
and, the proof is completed.
\end{proof}
\begin{theorem}[\textbf{Graded Betti numbers of path ideals of stars}]\label{graded betti numbers of stars} Let $\mathcal{S}_n$ be a star graph of size $n\geq 2$. For all $i\geq 1$ the nonzero graded Betti numbers of $S/I_2(\mathcal{S}_n)$ and $S/I_3(\mathcal{S}_n)$ are given by
\begin{equation*}
b_{i,j}(S/I_2(\mathcal{S}_n)) =
    \begin{cases}
       {n \choose j-1} & \text{ if } i=j-1  \\
      0 & \text{ otherwise } 
      \end{cases}
\end{equation*}
and,
\begin{equation*}
b_{i,j}(S/I_3(\mathcal{S}_n)) =
    \begin{cases}
       i{n \choose j-1} & \text{ if } i=j-2  \\
      0 & \text{ otherwise } 
      \end{cases}
\end{equation*}
where $j\leq n+1$. In particular, $S/I_t(\mathcal{S}_n)$ has a $(t-1)$-linear resolution.
\end{theorem}
\begin{proof}
Similar to proof of Corollary \ref{graded betti numbers of edge ideals of stars}.
\end{proof}

\end{document}